\documentclass[12pt]{amsart}

\usepackage{MnSymbol}
\newtheorem{theorem}{Theorem}[section]
\newtheorem{lemma}[theorem]{Lemma}
\newtheorem{prop}[theorem]{Proposition}

\theoremstyle{definition}

\newtheorem{ass}[theorem]{Assumption}

\theoremstyle{remark}
\newtheorem{remark}[theorem]{Remark}

\newcommand{\rr}{{\mathbb R}}

\newcommand{\N}{{\mathbb N}}

\usepackage{comment}
\usepackage{tikz}
\usepackage{float}
\usepackage{relsize}
\usepackage{enumitem}
\usepackage[T1]{fontenc}
\usepackage[english]{babel}
\usepackage{color}
\usepackage{graphicx}
\usepackage{cite}
\usepackage{amsmath}

\date{\today}

\begin{document}

\sloppy
\title[On mixed fractional SDEs with discontinuous drift coefficient]{On mixed fractional SDEs with discontinuous drift coefficient} 

\author{Ercan S\"onmez}
\address{Ercan S\"onmez, Department of Statistics, University of Klagenfurt, Universit\"atsstra{\ss}e 65--67, 9020 Klagenfurt, Austria}
\email{ercan.soenmez\@@{}aau.at} 

\begin{abstract}
We prove existence and uniqueness of the solution for a class of mixed fractional stochastic differential equations with discontinuous drift driven by both standard and fractional Brownian motion. Additionally, we establish a generalized It\^{o} rule valid for functions with absolutely continuous derivative and applicable to solutions of mixed fractional stochastic differential equations with Lipschitz coefficients, which plays a key role in our proof of existence and uniqueness. The proof of such a formula is new and relies on showing the existence of a density of the law under mild assumptions on the diffusion coefficient.
\end{abstract}

\maketitle

\allowdisplaybreaks
\section{Introduction}
Consider an autonomous stochastic differential equation (SDE)
\begin{align} \label{int1}
dX_t = a(X_t) dt + b(X_t) dW_t + c(X_t) dB_t^H,
\end{align}
with $a,b,c \colon \rr \to \rr$ being measurable functions, $W$ is a standard Brownian motion and $B^H$ is a fractional Brownian motion with Hurst index $H \in (\frac12,1)$.

If $c \equiv 0$ the corresponding SDE fits into the Markovian case and allows one to use the It\^{o} theory in order to investigate the SDE, see \cite{KS, LS, ZV, V2}  and references therein. Whereas if $b \equiv 0$ we find ourselves in the purely fractional case. With $H$ satisfying $H \in (\frac12 ,1)$ one can define stochastic integrals with respect to fractional Brownian motion utilizing a pathwise approach. A variety of methods have been developed and used in order to study such corresponding (stochastic) differential equations, see \cite{KK, KU, Ly, Lin, young, za, Nual, Nual1, Nual2, hu, Ruz}, in particular borrowing ideas and results from deterministic (geometric) differential equation theory. 

An eligible motivation to consider mixed SDEs emanates from applications in financial mathematics. Including both standard and fractional Brownian motion for the purpose of modeling randomness in a financial market enriches the model with flexibility, more particularly enables to capture and distinguish between two sources of randomness. Typically, standard Brownian motion models a white noise possessing no memory, whereas fractional Brownian motion models a noise having a long range property.

Questions regarding existence and uniqueness of a solution to mixed SDEs have been addressed in \cite{gu, KU2, MP, Mish, Mish2} under certain regularity assumptions on the coefficients $a,b,c$, see the aforementioned references and Section 2 for details. A main contribution of this work is the establishment of existence and uniqueness of solutions to mixed SDEs with irregular drift, in which the drift coefficient is allowed to be discontinuous.

In the Markovian case $c \equiv 0$ considerable effort has been made in the study of the corresponding SDE with discontinuous drift coefficient, see e.g. \cite{LS, LS2, ZV, TMG} and references therein, just to mention a few. Comparatively little is know for purely fractional SDEs, i.e. $b \equiv 0$, with discontinuous drift coefficient, see \cite{Fan, Suo, Mishbook} for the case $H > \frac12.$ In \cite[Theorem 3.5.14]{Mishbook} the existence of a strong solution is proven for purely fractional SDEs with additive noise, where the drift coefficient is given by the discontinuous function $a(x) = \operatorname{sign} (x)$ for all $x \in \rr$ and the Hurst index $H$ is restricted to $H \in (\frac12, \frac{1+\sqrt{5}}{4})$, see also \cite[Theorem 1]{Mishn} for a related result. In this paper, using a significantly different approach, we provide results on the existence and uniqueness of solutions to mixed SDEs, where we allow the drift coefficient to be in the more general class of piecewise Lipschitz continuous functions and we include all $H \in (\frac{1}{2},1)$.

The main obstacle one has to face in studying mixed SDEs arises from the fact that the two stochastic integrals involved in \eqref{int1} are of crucially different nature. The integral with respect to standard Brownian motion is a classical It\^{o} integral, the integral with respect to fractional Brownian motion is a pathwise Riemann-Stieltjes integral.

Let us outline the structure in achieving our main result. Whereas the proofs in \cite{Mish} use approximation theory and partially the approach in \cite{Nual2} we will borrow ideas from the purely non-fractional case and make them accessible in the setting of mixed SDEs. A key idea is to use a transformation technique originating from \cite{V1,V2, LS}. We will employ a transformation used in \cite{LS2}. More precisely, purposing to reduce the question to existence and uniqueness of solutions of classical equations the SDE is transformed in such a way that the discontinuity of the drift is removed, whereas the (regularity) properties of both the diffusion and fractional coefficient are preserved. This intention is, however, accompanied by challenges. In order to be able to apply the transform one needs to employ a generalized It\^{o} formula valid for convex functions with absolutely continuous derivative, which is not disposable so far. Therefor, another main contribution of this paper is to establish such a formula. Here we provide a novel proof, which is also new in the Markovian case. Our approach combines partially the procedure in \cite{Fan,LRS}. This is achieved by providing results on the absolute continuity of the law of mixed stochastic differential equations, which is then used in order to prove a variant of a generalized It\^{o} formula.

Here is a summary of the rest of this paper. In Section 2 we rigorously formulate the problem under consideration, also formulate and give a proof of our result on existence and uniqueness. As already mentioned, the proof invokes the It\^{o} formula generalized to convex functions. In Section 3 we first give a proof of a classical It\^{o} formula for mixed SDEs, in order to make the paper self-contained. Subsequent Section 4 is devoted to the study of the existence of a density of the law of mixed SDEs, which conclusively enables to prove a generalized It\^{o} formula.

\section{Existence and uniqueness}
Let $T \in (0, \infty)$ and $(\Omega, \mathcal{F}, (\mathbb{F}_t)_{t \in [0, T]}, P)$ be a filtered probability space satisfying the usual conditions. Suppose that $W=(W_t)_{t \in [0, T]}$ is a standard Brownian motion and $B^H=(B^H_t)_{t \in [0, T]}$ is a fractional Brownian motion with Hurst parameter $H \in (\frac{1}{2},1)$ independent of $W$, i.e. $B^H$ is a centered Gauss process with covariance function $R_H$ given by
$$R_H(t,s) = \mathbb{E} [ B_t^H B_s^H] = \frac{1}{2} \Big( |t|^{2H} + |s|^{2H} - |t-s|^{2H} \Big) , \quad t,s \in [0,T].$$
Let $a,b,c \colon  \rr \to \rr$ be measurable functions. We consider the following mixed stochastic differential equation
\begin{align} \label{sde1}
	 \begin{split}
		 X_t & = X_0+ \int_0^t a(X_s) ds + \int_0^t b(X_s) dW_s + \int_0^t c(X_s) dB_t^H , \quad t \in [0,T], \\
		 X_0 & = \xi \in \rr.
	 \end{split}
\end{align}
Under the assumptions that $a,b,c$ are Lipschitz and $c$ is differentiable with bounded and Lipschitz continuous derivative $c'$, it is known from \cite[Theorem 3.1]{Mish} that Equation \eqref{sde1} admits a unique solution. In fact, its proof shows that the following holds.

\begin{lemma} \label{th1}
	Assume that there exists $K \in (0, \infty)$ such that for all $x,y,x_1,x_2,x_3,x_4 \in \rr$
	\begin{align*}
		& | a(x) - a(y)| + | b(x) - b(y)| + | c(x) - c(y)| \leq K|x-y|, \\
		& | c(x_1) - c(x_2) - c(x_3) + c(x_4) | \leq K |x_1 - x_2 - x_3 + x_4| + K |x_1 - x_3| \Big( |x_1 - x_2| + |x_3 - x_4| \Big) .
	\end{align*}
	Then Equation \eqref{sde1} admits a unique solution.
\end{lemma}

First we show that Lemma \ref{th1} implies the following result, which has slightly weaker assumptions than \cite[Theorem 3.1]{Mish}.

\begin{prop} \label{th2}
	Assume that there exists $K \in (0, \infty)$ such that for all $x,y \in \rr$,
	\begin{align*}
	& | a(x) - a(y)| + | b(x) - b(y)| + | c(x) - c(y)| \leq K|x-y|.
	\end{align*}
	Moreover, assume that the function $c$ has a bounded Lebesgue-density $c' \colon \rr \to \rr$ which is Lipschitz continuous. Then Equation \eqref{sde1} admits a unique solution.
\end{prop}

\begin{proof}
	Let $x_1,x_2,x_3,x_4 \in \rr$ and $M_1, M_2, \ldots$ be unspecified constants. By assumption
	\begin{align*}
		c(x_1) - c(x_3) & = \int_{x_3}^{x_1} c' (u) du = \int_0^1 (x_1 - x_3) c' \big( \theta x_1 + (1-\theta) x_3 \big) d\theta .
	\end{align*}
	Thus
	\begin{align*}
		\begin{split}
			c(x_1) - c(x_2) - c(x_3) + c(x_4) & =  \int_0^1 (x_1 - x_3) c' \big( \theta x_1 + (1-\theta) x_3 \big) d\theta \\
			& \quad - \int_0^1 (x_2 - x_4) c' \big( \theta x_2 + (1-\theta) x_4 \big) d\theta \\
			& =  \int_0^1 (x_1 - x_2 - x_3 + x_4) c' \big( \theta x_2 + (1-\theta) x_4 \big) d\theta \\
			& \quad +  \int_0^1 (x_1 - x_3) \Big( c' \big( \theta x_1 + (1-\theta) x_3 \big) - c' \big( \theta x_2 + (1-\theta) x_4 \big) \Big) d\theta .
		\end{split}
	\end{align*}
	From this and the assumptions on $c'$ we obtain
	\begin{align*}
		| c(x_1) - c(x_2) - c(x_3) + c(x_4) | & \leq M_1 | x_1 - x_2 - x_3 + x_4 | \\
		& \quad + |x_1 - x_3| \cdot \int_0^1 M_2 | \theta x_1 + (1- \theta) x_3 - \theta x_2 - (1- \theta) x_4| d \theta \\
		& \leq M_3 | x_1 - x_2 - x_3 + x_4 | + M_4 |x_1 - x_3| \Big( |x_1 - x_2| + |x_3 - x_4| \Big) .
	\end{align*}
	Thus, Proposition \ref{th2} follows from Lemma \ref{th1}.
\end{proof}

Let us first recall the definition of the fractional integral appearing in equation \eqref{sde1}, which is an extension of the Stieltjes integral, see \cite{za1}. Let $a,b \in \rr$ with $a<b$. For $\alpha \in (0,1)$ and a function $f \colon [a,b] \to \rr$ the Weyl derivatives, denoted by $D_{a+}^\alpha f$ and $D_{b-}^\alpha f$, are defined by
\begin{align*}
	D_{a+}^\alpha f (x) & = \frac{1}{\Gamma (1- \alpha)} \Big( \frac{f(x)}{(x-a)^\alpha} + \alpha \int_a^x \frac{f(x)-f(y)}{(x-y)^{\alpha+1}} dy \Big) , \quad x\in (a,b), \\
	D_{b-}^\alpha f (x) &= \frac{(-1)^\alpha}{\Gamma (1- \alpha)} \Big( \frac{f(x)}{(b-x)^\alpha} + \alpha \int_x^b \frac{f(x)-f(y)}{(y-x)^{\alpha+1}} dy \Big) , \quad x\in (a,b),
\end{align*}
provided that $D_{a+}^\alpha f \in L_p$ and $D_{b-}^\alpha f \in L_p$ for some $p \geq 1$, respectively. The convergence of the above integrals at the singularity $y=x$ holds pointwise for almost all $x \in (a,b)$ if $p=1$ and in $L_p$-sense if $p \in (1, \infty)$. Denote by $g_{b-}$ the function given by
$$ g_{b-}(x) = g(x) - g(b-), \quad x \in (a,b).$$
Assume that $D_{a+}^\alpha f \in L_1$ and $D_{b-}^{1-\alpha} g_{b-} \in L_\infty$. Then the generalized Stieltjes or fractional integral of $f$ with respect to $g$ is defined as
$$\int_a^b f(x) dg(x) := (-1)^\alpha \int_a^b D_{a+}^\alpha f(x) D_{b-}^{1-\alpha} g_{b-} (x) dx.$$

	Let $\alpha \in (1-H, \frac{1}{2})$ and $\lambda \in (0,1]$. In the following denote by $W_0^{\alpha, \infty}$ the space of measurable functions $g \colon [0,T] \to \rr$ such that
\begin{align}\label{wnorm}
\| g\|_{\alpha, \infty} := \sup_{t \in [0,T]} \Big( |g(t)| + \int_0^t \frac{|g(t) - g(s)|}{(t-s)^{\alpha +1}} ds \Big) < \infty
\end{align}
and denote by $C^\lambda$ the space of $\lambda$-H\"older continuous functions $g \colon [0,T] \to \rr$ equipped with the norm
$$ \|g\|_\lambda :=  \sup_{t \in [0,T]}  |g(t)| + \sup_{0 \leq s < t \leq T}  \frac{|g(t) - g(s)|}{(t-s)^{\lambda}}.$$
It holds that for all $\varepsilon \in (0, \alpha)$
$$ C^{\alpha + \varepsilon} \subset W_0^{\alpha, \infty} \subset  C^{\alpha - \varepsilon}.$$
Let $f \in C^\lambda$ and $g \in C^\mu$ with $\lambda, \mu \in (0,1]$ such that $\lambda + \mu >1$. It is a well-known result, see \cite[Theorem 4.2.1]{za1}, that under this condition the fractional integral $\int_0^T f(x) dg(x)$ exists and agrees with the corresponding Riemann-Stieltjes integral. In particular, we have $D_{0+}^\alpha f \in L_1$ and $D_{T-}^{1-\alpha} g_{T-} \in L_\infty$.

Assume that $Y=(Y_t)_{t \in [0,T]}$ satisfies $Y \in C^\lambda$ almost surely for all $\lambda \in (0, \frac12)$. Then according to the aformentioned remarks the integral $\int_0^T c(Y_r) dB_r^H$ is well-defined, when $c \colon \rr \to \rr$ is Lipschitz continuous. Let $\lambda \in (0, \frac12)$ and $\beta \in (0,H)$ with $\lambda+\beta >1.$ Similarly to \cite[equation (3.8)]{Fan} one can prove the estimate
\begin{align} \label{neu1}
	\begin{split}
		\Big| \int_s^t c(Y_r) dB_r^H \Big| & \leq K \|B^H\|_\beta \Bigg( \int_s^t |c(X_r)| (r-s)^{-\alpha} (t-r)^{\alpha + \beta -1} dr \\
		& \quad + \|c\|_1 \| Y\|_\lambda \int_s^t (r-s)^{\lambda - \alpha} (t-r)^{\alpha + \beta -1} dr \Bigg)
	\end{split}
\end{align}
for $\alpha \in (1-H, \frac12)$ and some constant $K \in (0, \infty)$.

The goal of this paper is to study mixed SDEs with irregular coefficients. In particular the function $a$ is allowed to be discontinuous. It turns out that we can prove existence and uniqueness of such SDEs under the following assumption.
\begin{ass} \label{ass1} 
	\begin{itemize}
		\item[]
		\item [(A1)] The function $a$ is piecewise Lipschitz according to ~\cite[Definition 2.1]{LS} and its discontinuity points are given by $\xi_1 < \ldots < \xi_k \in \rr$ for some $k \in \N$, that is $a$ is Lipschitz continuous on each of the intervals $(-\infty, \xi_1)$, $(\xi_k, \infty)$ and $(\xi_j,\xi_{j+1})$, $1 \leq j \leq k-1$.
		\item [(A2)] The function $b$ is Lipschitz continuous on $\rr$ and $b (\xi_i) \neq 0$ for all $i \in \{1, \ldots, k\}$.
		\item [(A3)] The function $c$ is Lipschitz continuous with bounded derivative $c'$ which is Lipschitz continuous on $\rr$ as well and $c (\xi_i) = 0$ for all $i \in \{1, \ldots, k\}$.
	\end{itemize}
\end{ass}
We stress that the above assumptions are satisfied by a variety of (practical) examples. One of such is given by the SDE
$$ dX_t = -\operatorname{sign} (X_t) dt + dW_t + X_t dB_t^H, \quad X_0 = \xi \in \rr.$$

Our main Theorem now reads as follows.

\begin{theorem} \label{th3}
	Under Assumption \ref{ass1} Equation \eqref{sde1} admits a unique strong solution.
\end{theorem}

\begin{proof}
	Let $\Theta = \{ \xi_1, \ldots, \xi_k\}$ and $U = \rr \setminus \Theta$. Recall from \cite[Lemma 7]{TMG} that there is a function $G \colon \rr \to \rr$ satisfying the following:
	\begin{itemize}
		\item $G$ is Lipschitz continuous, differentiable on $\rr$ with $0 < \inf_{x \in \rr} G'(x) \leq \sup_{x \in \rr} G'(x) < \infty$
		\item $G$ has an inverse $G^{-1} \colon \rr \to \rr$ that is Lipschitz continuous and differentiable on $\rr$ with $G(\xi_i)=\xi_i$ for $i=1, \ldots, k$
		\item the derivative $G'$ of $G$ is Lipschitz continuous on $\rr$
		\item the derivative $G'$ of $G$ has a bounded Lebesgue density $G'' \colon \rr \to \rr$ that is piecewise Lipschitz with discontinuity points given by $\xi_1 < \ldots < \xi_k$ such that
		$$ \tilde{a} = (G' \cdot a + \frac{1}{2} G'' \cdot b^2) \circ G^{-1} \,\, \textnormal{and } \tilde{b} = (G' \cdot b) \circ G^{-1}$$
		are Lipschitz continuous.
	\end{itemize}
	Define for all $x \in \rr$ 
	\begin{align*}
		f(x) & := \frac{d}{dx} G' \big( G^{-1} (x) \big) = G'' \big( G^{-1} (x) \big) \cdot \frac{1}{G' \big( G^{-1} (x) \big)} \\
		& = h(x) \cdot j(x),
	\end{align*}
	with
	$$h(x) =  G'' \big( G^{-1} (x) \big) \quad \text{ and } \quad j(x) = \frac{1}{G' \big( G^{-1} (x) \big)}. $$
	By the assumptions listed above $f$ is bounded. Moreover, the function $j$ is bounded and differentiable (on $U$) with bounded derivative, thus the function $j$ is Lipschitz continuous.
	Since $h$ is bounded and Lipschitz continuous as well as a composition of Lipschitz continuous functions, the function $f$ is Lipschitz continuous on $U$ as a product of bounded and Lipschitz continuous functions. Similarly the function $g$ with
	\begin{align*}
	g(x) & := \frac{d}{dx} c \big( G^{-1} (x) \big) = c' \big( G^{-1} (x) \big) \cdot \frac{1}{G' \big( G^{-1} (x) \big)} , \quad x \in \rr,
	\end{align*}
	is bounded and Lipschitz continuous on $U$. Now, for all $x \in \rr$ let
	$$ \tilde{c} (x)  = c \big( G^{-1} (x) \big) \cdot  G' \big( G^{-1} (x) \big).$$
	Then $\tilde{c}$ is differentiable in $U$ and for $x \in U$ we have
	\begin{align*}
		\tilde{c} '(x) & =\frac{d}{dx} \tilde{c} (x) = c \big( G^{-1} (x) \big) \cdot f(x) + g(x) \cdot G' \big( G^{-1} (x) \big) \\
		& = c \big( G^{-1} (x) \big) \cdot G'' \big( G^{-1} (x) \big) \cdot \frac{1}{G' \big( G^{-1} (x) \big)} +  c' \big( G^{-1} (x) \big) \cdot \frac{1}{G' \big( G^{-1} (x) \big)} \cdot G' \big( G^{-1} (x) \big) .
	\end{align*}
	Then, by the considerations above the function $\tilde{c} '$ is bounded and Lipschitz continuous on $U$. Now consider the extension $\tilde{c}' \colon \rr \to \rr$ of $\tilde{c}'$, which we define by setting
	$$ \tilde{c} '(\xi_i) =  c \big( G^{-1} (\xi_i) \big) \cdot G'' \big( G^{-1} (\xi_i) \big) \cdot \frac{1}{G' \big( G^{-1} (\xi_i) \big)} + c' \big( G^{-1} (\xi_i) \big)$$
	for all $i \in \{1, \ldots, k\}$. By construction and by (A3) we have for all $i \in \{1, \ldots, k\}$
	\begin{align*}
		\tilde{c}' (\xi_i +) & = c (\xi_i + ) \cdot G'' (\xi_i+) \frac{1}{G'(\xi_i+)} + c' (\xi_i +) \\
		& = c' (\xi_i ) = c' (\xi_i -) = \tilde{c}' (\xi_i -).
	\end{align*}
	Thus, the function $\tilde{c}'$ is continuous and piecewise Lipschitz, hence Lipschitz continuous by \cite[Lemma 2.6]{LS}. Moreover, $\tilde{c}'$ is bounded, as $G''$ has compact support. We conclude that the function $\tilde{c}$ defined above admits a bounded Lebesgue-density that is Lipschitz continuous.
	
	Now consider the transformed process $(G(X_t))_{t \in [0,T]}$. By It\^{o}'s formula, Theorem \ref{genIto} below,
	$$dZ_t = \tilde{a} (Z_t) dt + \tilde{b} (Z_t) dW_t + \tilde{c} (Z_t) dB_t^H$$
	for $Z_t = G(X_t), t \in [0,T]$. By Proposition \ref{th2} the solution to this SDE is unique. This completes the proof.
\end{proof}

\section{It\^{o}'s formula for mixed SDEs}
In the following assume that $a,b,c \colon \rr \to \rr$ are Lipschitz, $c$ is differentiable with bounded and Lipschitz continuous derivative $c'$ so that, by \cite[Theorem 3.1]{Mish}, the solution to 
\begin{align*}
X_t & = X_0+ \int_0^t a(X_s) ds + \int_0^t b(X_s) dW_s + \int_0^t c(X_s) dB_t^H , \quad t \in [0,T], \\
X_0 & = \xi \in \rr,
\end{align*}
exists and is unique.
\begin{theorem} \label{Ito}
	Let $f \colon \rr \to \rr$ be twice continuously differentiable. Then, almost surely,
	\begin{align} \label{Itoformula}
		\begin{split}
			f (X_t) & = f(X_0) + \int_0^t f'(X_s) b(X_s) dW_s + \int_0^t f'(X_s) c(X_s) dB^H_s \\
			& \quad +  \int_0^t \Big( \frac{1}{2}b^2(X_s) f''(X_s) +a(X_s) f'(X_s) \Big) ds, \quad t \in [0, T] .
		\end{split}
	\end{align}
\end{theorem}

\begin{proof}
	By the usual localization argument, see the proof of \cite[Theorem 3.3]{KS}, we may assume that $f$ has compact support and that $f,f',f''$ are bounded. Fix $t \in (0, T]$ and a sequence $(\Pi^n = \{ 0=t_0^n < t^n_1< \ldots< t^n_m = t\})_{n \in \N}$, $m \in \mathbb{N}$, of partitions of $[0,t]$ with $\max_{1 \leq k \leq m} |t_k^n - t_{k-1}^n| \to 0$, $n \to \infty$. For notational simplicity we will suppress the index $n$ and simply write $\Pi = \{ 0=t_0 < t_1< \ldots< t_m = t\}$. By Taylor expansion
	\begin{align*}
		f(X_t) - f(X_0) & = \sum_{k=1}^{m} \Big( f(X_{t_k}) - f(X_{t_{k-1}}) \Big) \\
		& =  \sum_{k=1}^{m} f'(X_{t_{k-1}}) (X_{t_k} - X_{t_{k-1}}) + \frac{1}{2}\sum_{k=1}^{m} f''(\eta_{{k}})  (X_{t_k} - X_{t_{k-1}})^2
	\end{align*}
	with $\eta_{{k}}:= X_{t_{k-1}} + \theta_k (X_{t_k} - X_{t_{k-1}})$ for some random variable $\theta_k = \theta_k (\omega) \in [0,1], \omega \in \Omega$. We write
	\begin{equation*}
		f(X_t) - f(X_0) = J_0 + J_1 + J_2 + \frac{1}{2} J_3
	\end{equation*}
	with
	\begin{align*}
		J_0 &= \sum_{k=1}^m f'(X_{t_{k-1}}) \int_{t_{k-1}}^{t_k} a(X_s) ds, \\
		J_1 &= \sum_{k=1}^m f'(X_{t_{k-1}}) \int_{t_{k-1}}^{t_k} c(X_s) dB^H_s, \\
		J_2 &= \sum_{k=1}^m f'(X_{t_{k-1}}) \int_{t_{k-1}}^{t_k} b(X_s) dW_s, \\
		J_3 & = \sum_{k=1}^{m} f''(\eta_{{k}})  (X_{t_k} - X_{t_{k-1}})^2 .
	\end{align*}
	Observe that $J_0$ converges to the Lebesgue-Stieltjes integral $\int_0^t f'(X_s) a(X_s) ds$, as $n \to \infty$, almost surely. Now we turn to the term $J_1$. By \cite[Theorem 3.1]{Mish} for all $\alpha \in (1-H, \frac{1}{2})$ we have
	$$ X = (X_t)_{t \in [0,T]} \in W_0^{\alpha, \infty}$$
	almost surely. Therefore $X \in C^{\frac{1}{2}-}$ and we conclude that, by the assumptions made,
	$$ f'(X) c(X) \in C^{\frac{1}{2}-}$$
	almost surely. Thus, by \cite[Theorem 2.1.7]{Mishbook} the Riemann-Stieltjes integral
	$$ \int_0^t f'(X_s) c(X_s) dB_s^H$$
	exists and equals the limit $\lim_{n \to \infty} J_1$, almost surely. Now we consider the term $J_2$. Define $Y_s := f'(X_s)$, $s \in [0,T]$, which we are going to approximate by
	$$ Y_s^\Pi = f'(X_0) \mathbf{1}_{\{0\}} (s) + \sum_{k=1}^m f'(X_{t_{k-1}}) \mathbf{1}_{(t_{k-1}, t_k]} (s) , \quad s\in [0,t] .$$
	By It\^{o}'s isometry we have
	\begin{align*}
		\mathbb{E} \Big[ \Big( \int_0^t b(X_s) (Y_s^\Pi - Y_s) dW_s \Big) ^2 \Big] = \mathbb{E} \Big[  \int_0^t b^2(X_s) (Y_s^\Pi - Y_s)^2 ds  \Big] \to 0
	\end{align*}
	as $n \to \infty$, by the dominated convergence theorem. We obtain that
	$$ J_2 \to \int_0^t b(X_s) f'(X_s) dW_s, \quad n \to \infty,$$
	in $L^2$, i.e.
	$$ \mathbb{E} \Big[ \Big( J_2 - \int_0^t b(X_s) f'(X_s) dW_s \Big) ^2 \Big] \to 0, \quad n \to \infty.$$
	It remains to consider the expression $J_3$. We begin by writing
	$$ J_3 = J_4 + J_5 + J_6 + J_7 + J_8 + J_9$$
	with
	\begin{align*}
		J_4 &= \sum_{k=1}^m f''(\eta_{{k}})  \Big(\int_{t_{k-1}}^{t_k} c(X_s) dB^H_s\Big) ^2, \\
		J_5 &= \sum_{k=1}^m f''({\eta_{k}}) \Big( \int_{t_{k-1}}^{t_k} a(X_s) ds\Big) ^2, \\
		J_6 &= \sum_{k=1}^m f''({\eta_{k}}) \Big( \int_{t_{k-1}}^{t_k} b(X_s) dW_s\Big) ^2, \\
		J_7 & = 2\sum_{k=1}^{m} f''(\eta_{{k}}) \Big( \int_{t_{k-1}}^{t_k} a(X_s) ds\Big) \Big(\int_{t_{k-1}}^{t_k} c(X_s) dB^H_s\Big) ,\\
		J_8 & = 2\sum_{k=1}^{m} f''(\eta_{{k}}) \Big( \int_{t_{k-1}}^{t_k} a(X_s) ds\Big) \Big(\int_{t_{k-1}}^{t_k} b(X_s) dW_s\Big) ,\\
		J_9 & = 2\sum_{k=1}^{m} f''(\eta_{{k}}) \Big(\int_{t_{k-1}}^{t_k} c(X_s) dB^H_s\Big) \Big(\int_{t_{k-1}}^{t_k} b(X_s) dW_s\Big) .
	\end{align*}
	We first estimate $|J_4|$. In order to do this first define $W_T^{1-\alpha, \infty}$ for $\alpha \in (0, \frac{1}{2})$ as the space of measurable functions $g \colon [0,T] \to \rr$ such that
	$$ \| g \| _{1-\alpha, \infty, T} := \sup_{0 < s < t < T} \Big( \frac{|g(t) - g(s)|}{(t-s)^{1-\alpha }} + \int_s^t \frac{|g(y) - g(s)|}{(y-s)^{2-\alpha }} dy \Big) < \infty .$$
	We have the relation
	$$ C^{1-\alpha + \varepsilon} \subset W_T^{1-\alpha, \infty} \subset C^{1-\alpha}$$
	for every $\varepsilon \in (0, \infty)$. Recall that
	$$ (X_t)_{t \in [0,T]} \in W_0^{\alpha, \infty}$$
	almost surely. Since $c$ is assumed to be Lipschitz, we have that
	$$ \big( c(X_t)\big)_{t \in [0,T]} \in W_0^{\alpha, \infty}$$
	almost surely. Moreover, by the remarks from above
	$$ (B^H_t)_{t \in [0,T]} \in W_T^{H-\varepsilon, \infty} \quad \forall \varepsilon \in (0, \infty)$$
	almost surely. Consequently, by \cite[Proposition 4.2]{Nual2} we have
	$$\Big(  \int_0^t c(X_s) dB^H_s \Big)_{t \in [0,T]} \in C^{H - \varepsilon} \quad \forall \varepsilon \in (0, \infty)$$
	almost surely. Let $\gamma \in (\frac12 , H-\varepsilon)$. From this, together with the boundedness of $f''$, for some constant $K \in (0, \infty)$ we obtain
	$$ | J_4| \leq K \sum_{k=1}^{m} \Big( {t_k} - {t_{k-1}} \Big) ^{2\gamma} \to 0, \quad n \to \infty,$$
	almost surely. We continue by estimating $|J_5| + |J_7|$. Recall that the mapping
	$$ [0,T] \ni t \mapsto \int_0^t a(X_s) ds$$
	is continuous, of bounded variation and that the mapping
	$$ [0,T] \ni t \mapsto \int_0^t c(X_s) dB^H_s$$
	is continuous, almost surely. Thus
	\begin{align*}
		|J_5| + |J_7| & \leq K \Big( \max_{1 \leq k \leq m} |\int_{t_{k-1}}^{t_k} c(X_s) dB^H_s| +  \max_{1 \leq k \leq m}|\int_{t_{k-1}}^{t_k} a(X_s) ds| \Big) \\
		& \to 0, \quad n \to \infty,
	\end{align*}
	almost surely. The same argument shows that $|J_8|$ converges to 0, as $n \to \infty$, almost surely. Now we estimate $|J_9|$. Recall that
	$$\Big(  \int_0^t c(X_s) dB^H_s \Big)_{t \in [0,T]} \in C^{H - \varepsilon} \quad \forall \varepsilon \in (0, \infty)$$
	almost surely. We now combine this with the fact that
	$$\Big(  \int_0^t b(X_s) dW_s \Big)_{t \in [0,T]} \in C^{\frac{1}{2}-}$$
	almost surely. Indeed, choose $\alpha \in (0,H)$ and $\beta \in (0, \frac{1}{2})$ with $\alpha + \beta >1$. Employing the H\"older continuity we get for some constants
	\begin{align*}
		|J_9| & \leq K_1 \sum_{k=1}^m |t_k - t_{k-1}|^\alpha  |t_k - t_{k-1}|^\beta \\
		& \leq \max_{1 \leq k \leq m} |t_k - t_{k-1}|^{\alpha + \beta -1} K_1 \sum_{k=1}^m (t_k - t_{k-1}) \Big( \frac{t_k - t_{k-1}}{\max_{1 \leq k \leq m} |t_k - t_{k-1}|} \Big) ^{\alpha + \beta -1} \\
		& \leq \max_{1 \leq k \leq m} |t_k - t_{k-1}|^{\alpha + \beta -1} K_2 \sum_{k=1}^m (t_k - t_{k-1}) \\
		& \leq K_3 \max_{1 \leq k \leq m} |t_k - t_{k-1}|^{\alpha + \beta -1} \to 0, \quad n \to \infty,
	\end{align*}
	almost surely, since $\alpha + \beta >1$. It remains to consider
	$$J_6 = \sum_{k=1}^m f''(\eta_{k}) \Big( \int_{t_{k-1}}^{t_k} b(X_s) dW_s \Big) ^2.$$
	Define
	$$J'_6 = \sum_{k=1}^m f''(X_{t_{k-1}}) \Big( \int_{t_{k-1}}^{t_k} b(X_s) dW_s \Big) ^2.$$
	Then
	$$| J_6 - J'_6| \leq \max_{1 \leq k \leq m} |f''(\eta_{k}) - f''(X_{t_{k-1}})| \sum_{k=1}^m  \Big( \int_{t_{k-1}}^{t_k} b(X_s) dW_s \Big) ^2 .$$
	Note that $(b(X_s))_{s \in [0,T]}$ is adapted and continuous, hence progressively measurable. Moreover, since
	$$\mathbb{E} \Big[ \int_0^t b^2(X_s) ds \Big] < \infty$$
	by assumption, we get that the process
	$$\Big(  \int_0^t b(X_s) dW_s \Big)_{t \in [0,T]}$$
	is a martingale with respect to the underlying filtration. Thus, by \cite[Lemma 1.5.9]{KS}
	$$ \mathbb{E} \Big[ \sum_{k=1}^m \Big( \int_{t_{k-1}}^{t_k} b(X_s) dW_s \Big)^2 \Big] \leq K$$
	for some constant, and from the Cauchy-Schwarz-inequality we further get that
	\begin{align*}
		\mathbb{E} [| J_6 - J'_6|] \leq \sqrt{\mathbb{E} \Big[ \Big( \max_{1 \leq k \leq m} |f''(\eta_{k}) - f''(X_{t_{k-1}}) | \Big) ^2 \Big]} \to 0, \quad n \to \infty,
	\end{align*}
	by the dominated convergence theorem and the fact that $f'',X$ are continuous. Now define
	$$J_7' := \sum_{k=1}^m f''(X_{t_{k-1}})  \int_{t_{k-1}}^{t_k} b^2(X_s) ds .$$
	By the marintagle-property, see also \cite[p. 32]{KS}, we have
	\begin{align*}
		\mathbb{E} [| J'_6 - J'_7|^2] & = \mathbb{E} \Big[ | \sum_{k=1}^m f''(X_{t_{k-1}}) \Big( \big( \int_{t_{k-1}}^{t_k} b(X_s) dW_s \big) ^2 - \int_{t_{k-1}}^{t_k} b^2(X_s) ds  \Big) |^2 \Big] \\
		& = \mathbb{E} \Big[  \sum_{k=1}^m \big(f''(X_{t_{k-1}})\big) ^2 \Big( \big( \int_{t_{k-1}}^{t_k} b(X_s) dW_s \big) ^2 -  \int_{t_{k-1}}^{t_k} b^2(X_s) ds  \Big)^2 \Big] \\
		& \leq 2K \mathbb{E} \Big[  \sum_{k=1}^m \Big( \int_{t_{k-1}}^{t_k} b(X_s) dW_s \Big) ^4 +  \sum_{k=1}^m \Big(\int_{t_{k-1}}^{t_k} b^2(X_s) ds  \Big)^2 \Big] \\
		& \leq 2K \mathbb{E} \Big[  \sum_{k=1}^m \Big( \int_{t_{k-1}}^{t_k} b(X_s) dW_s \Big) ^4 +  \Big(\int_{0}^{t} b^2(X_s) ds  \Big) \max_{1 \leq k \leq m} \Big(\int^{t_k}_{t_{k-1}} b^2(X_s) ds  \Big)  \Big] \\
		& \to 0, \quad n \to \infty,
	\end{align*}
	by \cite[Lemma 1.5.10]{KS} and the dominated convergence theorem. Overall, we conclude that
	$$ J_3 \to \int_0^t f''(X_s) b^2(X_s) ds, \quad n \to \infty,$$
	in $L^1$. All these findings imply the result by standard arguments.
\end{proof}

\section{A generalized It\^{o} formula for mixed SDEs}

The goal of this section is a novel proof of the following new variant of It\^{o}'s formula differing from the one presented in the previous section.

\begin{theorem} \label{genIto}
	Let $\Theta = \{ \xi_1, \ldots, \xi_k\}$ be as in the proof of Theorem \ref{th3}. Let $f \colon \rr \to \rr$ be a continuously differentiable function such that for all $x \in \rr \setminus \Theta$ the second derivative $f''(x)$ exists and the function $f'' \colon \rr \setminus \Theta \to \rr$ is continuous and bounded. For definiteness extend $f''$ to $\rr$ in a way such that $f'' \colon \rr \to \rr$ is measurable. Moreover, let $X$ be as in Theorem \ref{Ito}, where we assume that $b(\xi_i) \neq 0$ for all $i=1, \ldots, k$. Then, almost surely,
	\begin{align*}
	\begin{split}
	f (X_t) & = f(X_0) + \int_0^t f'(X_s) b(X_s) dW_s + \int_0^t f'(X_s) c(X_s) dB^H_s \\
	& \quad +  \int_0^t \Big( \frac{1}{2}b^2(X_s) f''(X_s) +a(X_s) f'(X_s) \Big) ds, \quad t \in [0, T] .
	\end{split}
	\end{align*}
\end{theorem}

\begin{remark} \label{rem}
	In proving Theorem \ref{genIto} we shall often make use of the localization argument, as executed in the proof of Theorem \ref{Ito}. Let $X$ be as in Theorem \ref{genIto} and denote by $\| \cdot \|_t$, $t \in [0,T]$, a norm such that $\| X\|_t$ is almost surely finite. Choose a sequence of non-decreasing stopping times $(T_n)_{n \in \N}$ with the property that $\| X\|_t \leq n$ for all $t \in [0, T_n]$, $n \in \N$. Then it suffices to establish Theorem \ref{genIto} for the stopped process $X_t^{(n)} := X_{t \wedge T_n}$, $t \in [0,T]$, $n \in \N$. Therefore, without loss of generality in our proofs we will often assume that $\sup_{t \in [0,T]} \| X\|_t$ is bounded by some constant $K \in (0, \infty)$. A frequent choice for the norm will be $\| X\|_t = |X_t|$ or $\| X\|_t = \frac{|X_t - X_s|}{(t-s)^\lambda}$, $0 \leq s <t$, for $\lambda \in (0,1]$. 
\end{remark}

As mentioned in the introduction we present a proof of this variant of It\^{o}'s formula for mixed SDEs which combines ideas in \cite{Fan, LRS}. The first essential is to establish the existence of a density of the law of $X_t$ for every $t \in (0,T]$, where solely weak assumptions on the diffusion coefficient $b$ are imposed. In particular, we do not require non-degeneracy conditions, nor do we require any assumptions on the fractional coefficient $c$. 

We proceed in studying the existence of a density, which shortly after enables us to provide a proof of our main result in this section. We introduce some notation used throughout this section. For $h \in \rr$ and $m \in \N$ we define $\Delta_h$ to be the difference operator with respect to $h$ and $\Delta_h^m$ to be the difference operator of order $m$ which are given by
$$ \Delta_h f(x) = f(x+h) - f(x) \quad \text{ and } \quad \Delta_h^m f(x) = \Delta_h (\Delta^{m-1}_h f) (x) ,$$
respectively, for every function $f \colon \rr \to \rr$ and $x \in \rr$. Moreover, for $\gamma \in (0,m)$ we set $\mathcal{C}_b^\gamma$ to be the closure of bounded smooth functions with respect to the norm
$$ \| f\|_{\mathcal{C}_b^\gamma} := \|f \|_\infty + \sup_{ |h| \leq 1} \frac{\|\Delta_h^m f(x) \|_ \infty}{|h|^\gamma},$$
where $\| \cdot \|_\infty$ denotes the sup norm. Our second main result of this section reads as follows.

\begin{lemma} \label{d}
	Let $X$ be as in Theorem \ref{Ito}. For all $t \in (0,T]$ the law of $X_t$ admits a density with respect to the Lebesgue measure on the set $D_b = \{ x \in \rr: b(x) \neq 0\}$. In particular, it holds $P(X_t =x) = 0$ for all $x \in D_b$.
\end{lemma}

Let us mention that we could also prove that Lemma \ref{d} also holds for the set $\{ x \in \rr: c(x) \neq 0\}$ using similar arguments, but this will not be of importance for our purposes. The proof of Lemma \ref{d} closely follows the approach in \cite[Section 4]{Fan}. We will invoke the following two results. Lemma \ref{f2} is the statement of \cite[Lemma 4.5]{Fan} in dimension one, Lemma \ref{f1} is due to \cite[Section 2]{romito}.

\begin{lemma} \label{f2}
	Let $\rho \colon \rr \to [0, \infty)$ be a continuous function and $\delta \in (0, \infty)$. Denote $D_\delta = \{ x \in \rr : \rho(x) \leq \delta \}$ and define a function $h_\delta \colon \rr \to [0, \delta]$ with
	$$h_\delta (x) = \big( \inf \{ |x-z| : z \in D_\delta \} \big) \wedge \delta, \quad x \in \rr.$$
	Then $h_\delta$ has support in $\rr \setminus D_\delta$ and is globally Lipschitz continuous with Lipschitz constant 1. Moreover, for a probability measure $\mu$ on $\rr$, if for some $\delta >0$ the measure $\mu_\delta$ given by $\frac{d \mu_\delta}{d \mu} = h_\delta$ admits a density, then $\mu$ has a density on the set $\{ x \in \rr: \rho(x) >0 \}$.
\end{lemma}

\begin{lemma} \label{f1}
	Let $\mu$ be a finite measure on $\rr$. Assume that there exist $m \in \N$, $\gamma \in (0, \infty)$, $s \in (\gamma, m)$ and a constant $K \in (0, \infty)$ such that for all $\phi \in \mathcal{C}_b^\gamma$ and $h \in \rr$ with $|h | \leq 1$ it holds
	$$ \Big| \int_\rr \Delta_h^m \phi(x) d\mu (x) \Big| \leq K |h|^s \| \phi \|_{\mathcal{C}_b^\gamma} .$$
	Then $\mu$ has a density with respect to the Lebesgue measure on $\rr$.
\end{lemma}

In proving Lemma \ref{d} the goal is to apply Lemma \ref{f2} with $P_{X_t}$ and $\rho (x) = |b(x)|$, $x \in \rr$, where $P_{X_t}$ denotes the law of $X_t$. Lemma \ref{f1} will be used in order to deduce that the measure $\mu_\delta$ in Lemma \ref{f2} admits a density. First, we establish some auxiliary results. In what follows we denote
$$Y(\varepsilon) = X_{T-\varepsilon} + b(X_{T-\varepsilon}) (W_T - W_{T- \varepsilon}) + c(X_{T-\varepsilon}) (B_T^H - B^H_{T- \varepsilon})$$
for $\varepsilon \in (0,T)$. Let us now briefly recall some basic results concerning the representation of fractional Brownian motion in terms of Brownian motion, which will play a role in the remainder of this section. There exists a standard Brownian motion $B=(B_t)_{t \in [0,T]}$ such that
$$B_t^H = \int_0^t K_H(t,s) dB_s, \quad t \in [0,T],$$
where $K_H$ denotes the following square integrable kernel
$$K_H(t,s) = c_H \Big( \frac{t}{s} \Big)^{H-\frac12} (t-s)^{H-\frac12} - (H- \frac12)  s^{\frac12 - H} \int_s^t u^{H- \frac{3}{2}} (u-s)^{H-\frac12} du,$$
with $s \in (0,t)$ and some appropriate constant $c_H$, see \cite{nmall} for details. We will assume that the underlying filtration $\mathcal{F}$ is such that $B$ is $\mathcal{F}$-adapted. Moreover, the processes $W$ and $B$ are independent by assumption.

\begin{lemma} \label{f3}
	Let $$\xi = X_{T-\varepsilon} + c(X_{T-\varepsilon}) \int_0^{T- \varepsilon} \Big( K_H(T,s) - K_H(T-\varepsilon,s) \Big) dB_s$$ and $\eta = X_{T-\varepsilon}$. For all $u \in \rr$ we have
	\begin{align*}
		\mathbb{E} \Big[ \exp \big( iu Y(\varepsilon) \big) | \mathbb{F}_{T- \varepsilon}\Big] = \exp \Bigg( iu \xi - \frac12 u^2 \Big( b^2(\eta) \varepsilon + c^2(\eta) \int_{T- \varepsilon}^T K_H^2 (T,s) ds \Big) \Bigg) ,
	\end{align*}
	i.e. given $\mathbb{F}_{T- \varepsilon}$ the random variable $Y(\varepsilon)$ is conditionally Gaussian with mean $\xi$ and variance $b^2(\eta) \varepsilon + c^2(\eta) \int_{T- \varepsilon}^T K_H^2 (T,s) ds$.
\end{lemma}

\begin{proof}
	We note that both $W_T - W_{T- \varepsilon}$ and $\int_{T- \varepsilon}^T K_H (T,s) dB_s$ are independent of $\mathbb{F}_{T- \varepsilon}$, moreover $X_{T-\varepsilon}$ is $\mathbb{F}_{T- \varepsilon}$-measurable. Thus, from this we get
	\begin{align*}
		& \mathbb{E} \Bigg[ \exp \Big( iu b(X_{T-\varepsilon}) (W_T - W_{T- \varepsilon}) + iu c(X_{T- \varepsilon}) \int_{T- \varepsilon}^T K_H (T,s) dB_s \Big)  \Big| \mathbb{F}_{T- \varepsilon}\Bigg] \\
		& \quad = \mathbb{E} \Bigg[ \exp \Big( iu b(y) (W_T - W_{T- \varepsilon}) + iu c(y) \int_{T- \varepsilon}^T K_H (T,s) dB_s \Big)  \Bigg]_{\Big| y = \eta} \\
		& \quad = \exp \Bigg( - \frac{\big( b^2(y) \varepsilon + c^2(y) \int_{T- \varepsilon}^T K^2_H (T,s) ds \big)}{2} u^2 \Bigg) _{\Big| y = \eta}.
	\end{align*}
	From the integral representation of fractional Brownian motion it consequently follows
		\begin{align*}
	& 	\mathbb{E} \Big[ \exp \big( iu Y(\varepsilon) \big) | \mathbb{F}_{T- \varepsilon}\Big]  = \exp (iu \xi)\\
	& \quad \times \mathbb{E} \Bigg[ \exp \Big( iu b(\eta) (W_T - W_{T- \varepsilon}) + iu c(\eta) \int_{T- \varepsilon}^T K_H (T,s) dB_s \Big) \Big| \mathbb{F}_{T- \varepsilon} \Bigg] \\
	& \quad = \exp \Bigg( iu\xi - \frac{\big( b^2(\eta) \varepsilon + c^2(\eta) \int_{T- \varepsilon}^T K^2_H (T,s) ds \big)}{2} u^2 \Bigg) .
	\end{align*}
\end{proof}

\begin{lemma} \label{f4}
	Let $\beta \in (0,H)$ and $\beta' \in (0, \frac12)$. Then it holds
	$$\mathbb{E} [ |X_T - Y(\varepsilon)|] \leq K \Big( \varepsilon + \varepsilon^{\beta' + \beta} + \varepsilon^{\beta' + \frac12} \Big)$$
	for some constant $K\in (0, \infty)$.
\end{lemma}

\begin{proof}
	Throughout this proof we denote by $M_1, M_2, \ldots$ unspecified positive and finite constants. Recall that $\| X\|_{\beta'} < \infty$ almost surely. According to Remark \ref{rem} we may assume that $\| X\|_{\beta'} \leq M_1$. We have
	\begin{align*}
		|X_T - Y(\varepsilon)| = \Big| \int_{T- \varepsilon}^T a(X_r) dr + \int_{T- \varepsilon}^T \big( c(X_r) - c(X_{T- \varepsilon}) \big) dB_r^H + \int_{T- \varepsilon}^T \big( b(X_r) - b(X_{T- \varepsilon}) \big) dW_r \Big| .
	\end{align*}
	Due to our assumptions and the Lipschitz continuity of $a$ we obtain the estimate
	\begin{align} \label{pn1}
		\big| \int_{T- \varepsilon}^T a(X_r) dr \Big| \leq M_2 \varepsilon.
	\end{align}
	Furthermore, according to equation \eqref{neu1} we estimate
	\begin{align} \label{pn2}
		\begin{split}
		& \Big| \int_{T- \varepsilon}^T \big( c(X_r) - c(X_{T- \varepsilon}) \big) dB_r^H \Big| \\
		&  \leq M_3 \| B^H\|_\beta \Big( \int_{T- \varepsilon}^T |c(X_r) - c(X_{T-\varepsilon})| (r-T+\varepsilon)^{-\alpha} (T-r)^{\alpha + \beta -1} dr \\
		& \quad + M_4 \int_{T- \varepsilon}^T (r-T + \varepsilon)^{\beta' - \alpha} (T-r)^{\alpha + \beta -1} dr \Big) \\
		& \leq M_5 \| B^H\|_\beta \Big( \int_{T- \varepsilon}^T (r-T+\varepsilon)^{\beta'-\alpha} (T-r)^{\alpha + \beta -1} dr + M_6 \varepsilon^{\beta' + \beta} \Big) \\
		& \leq M_7 \| B^H\|_\beta \varepsilon^{\beta' + \beta}.
		\end{split}
	\end{align}
	In addition, the Burkholder-Davis-Gundy inequality gives us
	\begin{align} \label{pn3}
	\begin{split}
	& \mathbb{E}\Big[ \Big| \int_{T- \varepsilon}^T \big( b(X_r) - b(X_{T- \varepsilon}) \big) dW_r \Big| \Big]  \leq M_8  \mathbb{E}\Big[ \Big( \int_{T- \varepsilon}^T \big| b(X_r) - b(X_{T- \varepsilon})\big| ^2  dr \Big) ^{\frac12} \Big] \\
	& \leq M_9 \mathbb{E}\Big[ \Big( \int_{T- \varepsilon}^T (r-T + \varepsilon)^{2\beta'}  dr \Big) ^{\frac12} \Big] \leq M_{10} \varepsilon^{\beta' + \frac12}.
	\end{split}
	\end{align}
	Recall that there is a random variable $A$ with finite moments of every order such that
	$$ | B_v^H -B_u^H | \leq A |v-u| ^\beta  \quad \forall u,v \in [0,T]$$
	almost surely, see e.g. \cite[Lemma 7.4]{Nual2}. Using this, it is easy to see that for every $k \in \N$
	$$ \mathbb{E} \Big[ \| B^H \|_{\beta}^k \Big] \leq M_{11}.$$
	Combining this with \eqref{pn1}, \eqref{pn2} and \eqref{pn3} completes the proof.
\end{proof}

\noindent \textit{Proof of Lemma \ref{d}.} Without loss of generality we only prove that the law of $X_T$ is absolutely continuous on the set $D_b$. The goal is to apply Lemma \ref{f2}. To this end, define the function $\rho \colon \rr \to [0, \infty)$ to be $\rho(x) = |b(x)|$, $x \in \rr$, and the measure $\mu_\delta$ given by $d\mu_\delta (z) = h_\delta (z) dP_{X_T} (z)$. It suffices to prove that $\mu_\delta$ admits a density with respect to the Lebesgue measure, and we will make use of Lemma \ref{f1} in order to show this. According to the latter, it suffices to find $m \in \N$, $\gamma \in (0, \infty)$, $s \in (\gamma, m)$ such that
\begin{equation} \label{dp1}
	\Big| \mathbb{E} \Big[ h_\delta (X_T) \Delta_h^m \phi (X_T) \Big] \Big| \leq K |h|^s \| \phi\|_{\mathcal{C}_b^\gamma} 
\end{equation}
for all $h \in \rr$ with $|h| \leq 1$, for all $\phi \in \mathcal{C}_b^\gamma$ and some constant $K \in (0, \infty)$. The specific choice of $m$, $\gamma$ and $s$ will be given later at the end of this proof. In the following, as before we denote by $M_1, M_2, \ldots$ unspecified positive and finite constants. Using the notation and results of Lemma \ref{f2} and Lemma \ref{f4} we estimate
\begin{align} \label{dp2}
\begin{split}
	& \Big| \mathbb{E} \Big[ h_\delta (X_T) \Delta_h^m \phi (X_T) \Big] \Big|  \leq 	\Big| \mathbb{E} \Big[ \Big( h_\delta (X_T) - h_\delta (X_{T-\varepsilon}) \Big) \Delta_h^m \phi (X_T) \Big] \Big| \\
	& \quad + \Big| \mathbb{E} \Big[ h_\delta (X_{T-\varepsilon}) \Big( \Delta_h^m \phi (X_T) -  \Delta_h^m \phi \big( Y (\varepsilon) \big) \Big) \Big] \Big| + \Big| \mathbb{E} \Big[ h_\delta (X_{T-\varepsilon}) \Delta_h^m \phi \big( Y (\varepsilon) \big) \Big] \Big| \\
	& \leq M_1 \| \phi\|_{\mathcal{C}_b^\gamma}  |h|^\gamma \mathbb{E} [ |X_T - X_{T-\varepsilon}|] + M_2 \| \phi\|_{\mathcal{C}_b^\gamma}  \mathbb{E} [ |X_T - Y(\varepsilon)|]^\gamma \\
	& \quad + \Big| \mathbb{E} \Big[  h_\delta (X_{T-\varepsilon}) \mathbb{E} \big[ \Delta_h^m \phi \big( Y (\varepsilon) \big) \big| \mathbb{F}_{T-\varepsilon} \big] \Big] \Big| \\
	& \leq M_3 \| \phi\|_{\mathcal{C}_b^\gamma}  \Big( |h|^\gamma \varepsilon^{\beta'} + \big( \varepsilon + \varepsilon^{\beta' + \beta} + \varepsilon^{\beta'+\frac12} \big)^\gamma \Big) + \Big| \mathbb{E} \Big[  h_\delta (X_{T-\varepsilon}) \mathbb{E} \big[ \Delta_h^m \phi \big( Y (\varepsilon) \big) \big| \mathbb{F}_{T-\varepsilon} \big] \Big] \Big| .
\end{split}
\end{align}
Recall that, by Lemma \ref{f3}, $Y(\varepsilon) | \mathbb{F}_{T-\varepsilon} \sim \mathcal{N} (\xi, \sigma^2 (\eta))$ with
$$\sigma^2 (y) = b^2(y) \varepsilon + c^2(y) \int_{T- \varepsilon}^T K_H^2 (T,s) ds , \quad y \in \rr,$$
and that $h_\delta (y) = 0$ for all $y \in \rr$ with $|b(y)| \leq \delta$. Let $p_y \colon \rr \to \rr$ be the density of a Gaussian distribution with mean zero and variance $\sigma^2 (y)$. Then
\begin{align*}
	\sup_{y \in \rr : |b(y)| \geq \delta} \int_\rr \Big| \frac{d^k}{dz^k} p_y(z) \Big| dz & \leq M_4 \sup_{y \in \rr : |b(y)| \geq \delta} (\sigma^2)^{-\frac{k}{2}} \leq M_4 \delta^{-k} \varepsilon^{-\frac{k}{2}}
\end{align*}
for every $k \in \N$. From this we obtain
\begin{align*}
	\Big| \mathbb{E} \Big[  h_\delta (X_{T-\varepsilon}) \mathbb{E} \big[ \Delta_h^m \phi \big( Y (\varepsilon) \big) \big| \mathbb{F}_{T-\varepsilon} \big] \Big] \Big| \leq M_5  \| \phi\|_{\mathcal{C}_b^\gamma}  |h|^m \varepsilon^{-\frac{m}{2}}.
\end{align*}
Combining the latter inequality with \eqref{dp2} gives us
\begin{align*}
	\Big| \mathbb{E} \Big[ h_\delta (X_T) \Delta_h^m \phi (X_T) \Big] \Big|  & \leq M_6 \| \phi\|_{\mathcal{C}_b^\gamma} \Big( |h|^\gamma \varepsilon^{\beta'} + \big( \varepsilon + \varepsilon^{\beta' + \beta} + \varepsilon^{\beta'+\frac12} \big)^\gamma + |h|^m \varepsilon^{-\frac{m}{2}} \Big) \\
	& \leq M_7 \| \phi\|_{\mathcal{C}_b^\gamma} \Big( |h|^\gamma \varepsilon^{\beta'} +  \varepsilon^{(\beta'+\frac12) \gamma}  + |h|^m \varepsilon^{-\frac{m}{2}} \Big) .
\end{align*}
Now it is  not difficult to see that the assumptions of Lemma \ref{f1} are satisfied. For example we can set $m=4$, $\gamma=1$, $\varepsilon = |h|^{\frac{4}{3}}$ and choose $\beta' \in (0, \frac12)$ such that $(\beta' + \frac12) \frac{4}{3} >1$. Then it is easy to see that \eqref{dp1} holds and the proof is complete.  \hfill  {\Large $\Box$ } \newline

Now we are in position to prove Theorem \ref{genIto}. \newline

\noindent \textit{Proof of Theorem \ref{genIto}.} The proof borrows ideas from the theory of approximate identities as outlined in \cite[Section A]{LRS}. Without loss of generality we assume that $f$ has compact support, implying in our case that $f$, $f'$ and $f''$ are bounded, and we assume that $f''$ has one discontinuity point at $\xi_1=0$. According to \cite[Lemma 1-3]{LRS} one can choose a sequence $(\phi_n)_{n \in \N}$ of twice continuously differentiable non-negative functions such that for $f_n := f * \phi_n$, $n \in \N$,
$$ \lim_{n \to \infty} \| f - f_n \|_\infty + \| f' - f_n' \|_\infty = 0,$$
and 
$$ \lim_{n \to \infty} | f''(x) - f_n''(x) | = 0$$
for all continuity points $x$ of $f''$, and $\| f_n''\|_\infty \leq K$ for some constant $K \in (0, \infty)$ independent of $n \in \N$. Let $t \in [0,T]$. By Theorem \ref{Ito}
\begin{align} \label{pgi1}
\begin{split}
	f_n (X_t) & = f_n(X_0) + \int_0^t f_n'(X_s) b(X_s) dW_s + \int_0^t f_n'(X_s) c(X_s) dB^H_s \\
	& \quad +  \int_0^t \Big( \frac{1}{2}b^2(X_s) f_n''(X_s) +a(X_s) f_n'(X_s) \Big) ds.
\end{split}
\end{align}
By our assumptions we obtain convergence
\begin{align*}
	\lim_{n \to \infty} \int_0^t f_n'(X_s) b(X_s) dW_s & = \int_0^t f'(X_s) b(X_s) dW_s, \\	\lim_{n \to \infty} \int_0^t f_n'(X_s) a(X_s) ds & = \int_0^t f'(X_s) a(X_s) ds,
\end{align*}
where the first convergence is uniformly in probability and the second convergence holds almost surely. Now we turn to the fractional integral in \eqref{pgi1}. By assumption
$$ f'(X) c(X) \in C^{\frac12 -},$$
so that the fractional integral
$$\int_0^t f'(X_s) c(X_s) dB_s^H$$
exists and agrees with the Riemann-Stieltjes integral. In particular, we have $D_{0+}^\alpha f'(X) c(X) \in L_1$, where $\alpha \in (1-H, \frac12)$, and $D_{0+}^\alpha f_n'(X) c(X)$ converges in $L_1$ to $D_{0+}^\alpha f'(X) c(X)$. Thus, by dominated convergence and the definition of the fractional integral
\begin{align*}
	\lim_{n \to \infty} \int_0^t f_n'(X_s) c(X_s) dB_s^H & = \lim_{n \to \infty} (-1)^\alpha \int_0^t D_{0+}^\alpha \big( f_n'(X) c(X) \big) (s)  D_{t-}^{1-\alpha} B_{t-}^H (s) ds \\
	& =  (-1)^\alpha \int_0^t D_{0+}^\alpha \big( f'(X) c(X) \big) (s)  D_{t-}^{1-\alpha} B_{t-}^H (s) ds \\
	& = \int_0^t f'(X_s) c(X_s) dB_s^H .
\end{align*}
It remains to consider the term in \eqref{pgi1} with the second derivative. In order to prove its convergence we will make use of Lemma \ref{d}. We write
\begin{align*}
\frac12 \int_0^t f_n''(X_s) b^2(X_s) ds = \frac12 \int_0^t \mathbf{1}_{\{ |X_s| > \frac{1}{n} \}}f_n''(X_s) b^2(X_s) ds + \frac12 \int_0^t \mathbf{1}_{\{ |X_s| \leq \frac{1}{n} \}}f_n''(X_s) b^2(X_s) ds .
\end{align*}
From dominated convergence we derive
$$ \lim_{n \to \infty} \frac12 \int_0^t \mathbf{1}_{\{ |X_s| > \frac{1}{n} \}}f_n''(X_s) b^2(X_s) ds = \frac12 \int_0^t \mathbf{1}_{\{ X_s \neq 0 \}}f''(X_s) b^2(X_s) ds.$$
From Fatou's Lemma, for every sequence $(h_n)_{n \in \N}$ of measurable, non-negative and bounded functions with $\| h_n \|_\infty \leq K$ for all $n \in \N$ and some constant $K \in (0, \infty)$, we obtain
 \begin{align*}
	 \limsup_{n \to \infty} \int_0^t \mathbf{1}_{\{ |X_s| \leq \frac{1}{n} \}}h_n(X_s) b^2(X_s) ds & \leq \int_0^t  \limsup_{n \to \infty} \mathbf{1}_{\{ |X_s| \leq \frac{1}{n} \}}h_n(X_s) b^2(X_s) ds \\
	 & \leq \int_0^t \mathbf{1}_{\{ |X_s| = 0 \}} \limsup_{n \to \infty} h_n(X_s) b^2(X_s) ds \\
	 & \leq K \int_0^t \mathbf{1}_{\{ X_s = 0 \}} b^2(X_s) ds.
 \end{align*}
Now, by Lemma \ref{d}
$$ \mathbb{E} \Big[ \int_0^t \mathbf{1}_{\{ X_s = 0 \}} b^2(X_s) ds \Big] = b^2 (0) \int_0^t P(X_s=0) ds = 0,$$
which yields that
$$\int_0^t \mathbf{1}_{\{ X_s = 0 \}} b^2(X_s) ds =0$$
almost surely. Therefore
$$ \limsup_{n \to \infty} \int_0^t \mathbf{1}_{\{ |X_s| \leq \frac{1}{n} \}}h_n(X_s) b^2(X_s) ds = 0,$$
and in case of the particular choice $h_n = |f_n'' - f''|$
$$ \limsup_{n \to \infty} \int_0^t \mathbf{1}_{\{ |X_s| \leq \frac{1}{n} \}} \Big| f_n'' (X_s) - f''(X_s) \Big|  b^2(X_s) ds = 0.$$
Finally, we get that
\begin{align*}
 	& \Big| \frac12 \int_0^t f_n''(X_s) b^2(X_s) ds - \frac12 \int_0^t f''(X_s) b^2(X_s) ds \Big| \\
 	& \leq \frac12 \Big| \int_0^t\mathbf{1}_{\{ |X_s| > \frac{1}{n} \}} \Big( f_n''(X_s) - f''(X_s) \Big) b^2(X_s) ds \Big| +  \frac12 \int_0^t\mathbf{1}_{\{ |X_s| \leq \frac{1}{n} \}}  \Big| f_n''(X_s) - f''(X_s) \Big| b^2(X_s) ds \\
 	& \to 0,
\end{align*}
as $n \to \infty$. Overall, we have shown that the assertion follows by letting $n \to \infty$ in \eqref{pgi1}.  \hfill  {\Large $\Box$ }

\subsection*{Acknowledgement} The author would like to thank David Nualart for some interesting discussions.

\bibliographystyle{amsplain}
\bibliography{lit}

\end{document}